\documentclass[12pt]{amsart}
\usepackage{amssymb,amsmath,amsthm}
\usepackage[usenames]{color}
\def\S{\mathbb{S}}

\newtheorem{theo}{\bf{Theorem}}[section]

\newtheorem{prob}[theo]{Problem}

\theoremstyle{definition}
\newtheorem{defi}[theo]{Definition}
\newtheorem{rem}[theo]{Remark}

\begin{document}
\title{Bipartite Minors}

\author{Maria Chudnovsky}
\address{Department of Industrial Engineering and Operations Research,
Columbia University, New York, NY 10027, USA}
\email{mchudnov@columbia.edu}

\author{Gil Kalai}
\address{{Institute of Mathematics, Hebrew University of Jerusalem,
Jerusalem 91904, Israel}
{\tiny and}
{Department of Computer Science and  Department of Mathematics,
Yale University New Haven, CT 06511, USA}}
\email{kalai@math.huji.ac.il}

\author{Eran Nevo}
\address{
Department of Mathematics,
Ben Gurion University of the Negev,
Be'er Sheva 84105, Israel
}
\email{nevoe@math.bgu.ac.il}

\author{Isabella Novik}
\address{
Department of Mathematics, Box 354350,
University of Washington, Seattle, WA 98195-4350, USA
}
\email{novik@math.washington.edu}

\author{Paul Seymour}
\address{Department of Mathematics, Princeton University,
Washington Rd, Princeton, NJ 08544, USA}
\email{pds@math.princeton.edu}

\thanks{Research of the first author was partially supported by NSF Grant
DMS-1265803, of the second author by ERC advanced grant 320924,
ISF grant 768/12, and NSF grant DMS-1300120,
of the third author by Marie Curie grant IRG-270923 and ISF grant 805/11,
of the fourth author by NSF grant DMS-1069298, and of the fifth author by
ONR grant N00014-10-1-0680 and NSF grant DMS-1265563.}

\begin{abstract}
We introduce a notion of bipartite minors and prove a bipartite
analog of Wagner's theorem: a bipartite graph is planar if and only if it
does not contain $K_{3,3}$ as a bipartite minor. Similarly, we provide a forbidden
minor characterization for outerplanar graphs and forests. We then establish a
recursive characterization of bipartite $(2,2)$-Laman graphs --- a certain
family of graphs that contains all maximal bipartite planar graphs.
\end{abstract}

\maketitle

\section{Introduction}
Wagner's celebrated theorem \cite{Wagner:Kuratowski},
\cite[Theorem 4.4.6]{Diestel} provides a characterization
of planar graphs in terms of minors: a graph $G$ is planar if and only if
it contains neither $K_5$ nor $K_{3,3}$ as a minor. Unfortunately, a minor
of a bipartite graph is not always bipartite as contracting edges destroys
2-colorability. Here, we introduce a notion of a {\em bipartite minor}: an
operation that applies to bipartite graphs and outputs bipartite graphs.
We then prove a bipartite analog of Wagner's theorem:
a bipartite graph is planar if and only if it does not contain $K_{3,3}$
as a bipartite minor. Similarly, we provide a forbidden bipartite minor
characterization for bipartite outerplanar graphs and forests.

All the graphs considered in this note are simple graphs. A graph
with vertex set $V$ and edge set $E$ is denoted by $G=(V,E)$.
We denote the edge connecting
vertices $i$ and $j$ by $ij$. %(instead of $\{i, j\}$).
A graph is {\em bipartite} if there exists a bipartition
(or bicoloring in red and blue)
of the vertex set $V$ of $G$, $V=A\uplus B$, in such a way that no two
vertices from the same part are connected by an edge.
When discussing bipartite graphs,
we fix such a bipartition and write $G=(A\uplus B,E)$; we refer to $A$
and $B$ as {\em parts} or {\em sides} of $G$.

As bipartite planar graphs with $n\geq 3$ vertices have at most $2n-4$ edges,
and as all their subgraphs are also bipartite and planar, and hence satisfy the
same restriction on the number of edges, it is natural to consider the family
of maximal bipartite graphs possessing this property. Specifically, we say that
a bipartite graph $G=(A\biguplus B, E)$ with $|A|\geq 2$
and $|B|\geq 2$ is {\em $(2,2)$-Laman} if (i) $G$
has exactly $2(|A|+|B|)-4$ edges, and
(ii) every subgraph $H$ of $G$ with at least 3 vertices
has at most $2|V(H)|-4$ edges.
%$(A'\biguplus B', E')$ of $G$ has at most $\max\{0,2(|A'|+|B'|)-4\}$
%edges (here $A'\subseteq A$ and $B' \subseteq B$).
Note that the family of $(2,2)$-Laman graphs is strictly larger than that
of maximal bipartite planar graphs: indeed, taking $n\geq 2$ copies of
$K_{3,3}$ minus an edge, and gluing all these copies together along the
two vertices of the missing edge, produces a graph on $4n+2$ vertices
with $8n$ edges; this graph is $(2,2)$-Laman, but it is not planar.

Our second main result is a recursive characterization of $(2,2)$-Laman graphs.
We remark that the name $(2,2)$-Laman is motivated by Laman's
theorem \cite{Laman} from rigidity theory of graphs, and
its relation to a recent theory of rigidity for bipartite graphs can be
found in \cite{BipRig}. As such, this paper is a part of a project
to understand notions of minors and graph-rigidity for bipartite
graphs as well as to understand higher-dimensional generalizations.

The rest of this note is organized as follows: in Section 2 we define
bipartite minors and prove the bipartite analog of Wagner's theorem
and analogous theorems for bipartite outerplanar graphs and forests
(deferring treatment of some of the cases to the Appendix). Then in
Section 3 we discuss $(2,2)$-Laman graphs.

\section{Wagner's theorem for bipartite graphs}
We start by defining a couple of basic operations on (bipartite) graphs.
If $G=(V,E)$ is a graph and $v$ is a vertex of $G$, then $G-v$ denotes
the induced subgraph of $G$ on the vertex set $V-\{v\}$.
If $G=(V=A\uplus B,E)$ is a bipartite graph and $u,v$ are two vertices
from the same part, then the {\em contraction} of $u$ with $v$ is a graph
$G'$ on the vertex set $V-\{u\}$ obtained from $G$ by identifying $u$ with
$v$ and deleting the extra copy from each double edge that was created.
Observe that $G'$ is also bipartite.

Recall that if $G$ is a graph and $C$ is a cycle of $G$, then $C$
is {\em non-separating} if the removal of the vertices of $C$ from
$G$ does not increase the number of connected components.
A cycle $C$ is {\em induced} (or chordless) if each two nonadjacent
vertices of $C$ are not connected by an edge in $G$. Induced non-separating cycles
are known in the literature as {\em peripheral} cycles.

We now come to the main definition of this section.
\begin{defi} \label{def:bip-minor}
Let $G$ be bipartite graph. We say that a graph $H$ is a
{\em bipartite minor of} $G$, denoted $H <_b G$,
if there is a sequence of graphs $G=G_0,G_1,\ldots,G_t=H$
where for each $i$, $G_{i+1}$ is obtained from $G_i$ by either
deletion (of a vertex or an edge) or  \emph {admissible contraction}.
A contraction of a vertex $u$ with a vertex $v$ in $G_i$ is called \emph{admissible}
if $u$ and $v$ have a common neighbor in $G_i$,
%(and hence belong to the same part of $G_i$),
and at least one of these common neighbors,
say $w$, is such that the path $(u,w,v)$ is a part of a peripheral
cycle in $G_i$.
\end{defi}
For instance, applying an admissible contraction to an $8$-cycle
results in a $6$-cycle plus an edge attached to this cycle at one vertex.
Note that since each admissible contraction identifies two vertices that have
a common neighbor, these two vertices are from the same part of the graph. Thus
all bipartite minors are bipartite graphs. The importance of the notion
of bipartite minors is explained by the following result that can be
considered as a bipartite analog of Wagner's theorem.

\begin{theo}\label{thm:bipWagner}
A bipartite graph $G$ is planar if and only if $G$ does not
contain $K_{3,3}$ as a bipartite minor.
\end{theo}
\begin{proof} First assume that $G$ is planar.
We may assume that $G$ is connected.
To verify that $G$ does not
contain $K_{3,3}$ as a bipartite minor, it suffices to show that deletions and
admissible contractions preserve planarity.  This is clear for deletions.
For admissible contractions, consider an embedding
of $G$ in a 2-sphere $\S^2$, and let $C$ be a peripheral cycle of
$G$ that contains a path $(u,w,v)$. By the Jordan-Sch\"onflies theorem,
the complement of the image of $C$ in $\S^2$ consists of
two components, each homeomorphic to an open 2-ball.
As $C$ is peripheral, one
of these components contains no vertices/edges of $G$,
and hence is a face of the embedding of $G$. Contracting $u$
with $v$ ``inside this face'' produces an embedding of the
resulting graph in $\S^2$.

Assume now that $G$ is not planar. We must show that $G$ contains $K_{3,3}$
as a bipartite minor. By  Kuratowski's theorem \cite[Theorem 4.4.6]{Diestel},
$G$ contains a subgraph $H$ that is a subdivision of either $K_5$ or
$K_{3,3}$. Hence, it only remains to show that $K_{3,3} <_b H$. We first treat
the case where an edge $e$ of the original $K_5$ (or $K_{3,3}$) is subdivided at
least twice. Let $C$ be a peripheral cycle of the original
$K_5$ (or $K_{3,3}$) that contains $e$ (there exists such $C$ --- a
$3$-cycle for $K_5$ and a $4$-cycle for $K_{3,3}$), let $C'$ be the subdivision of
$C$ in $H$, and let $(a,b,a')$ be a path of length two in $H$ that
is contained in $e$. Then $(a,b,a')$ is a part of a peripheral cycle $C'$
in $H$, and hence contracting $a'$ with $a$ is an admissible contraction in $H$.
Performing this contraction and then deleting $b$, we obtain
a new bipartite subdivision $H'$ of $K_5$ (or $K_{3,3}$) that
 subdivides $e$ with two fewer interior vertices than $H$, but
agrees with $H$ on all other edges of $K_5$ ($K_{3,3}$, respectively).
Thus, we can assume that each edge of the original $K_5$ (or $K_{3,3}$)
is subdivided at most once. By symmetry, this reduces the
problem of finding $K_{3,3}$ as a bipartite minor of $H$ to finding $K_{3,3}$
as a bipartite minor of the nine bipartite graphs described below.
These cases are treated in the Appendix.

There are three bipartite graphs which are subdivisions of $K_5$ to consider,
denoted by $G_{(i)}$ for $i=5,4,2$, where $G_{(i)}$ is the graph obtained from $K_5$
by coloring $i$ of its vertices red, the other $5-i$ blue, then subdividing
each monochromatic edge once, and coloring the subdivision vertex red/blue
so that its color is opposite to that of the vertices of the original monochromatic edge.
For example, $G_{(5)}$ is the barycentric subdivision of $K_5$, endowed with a $2$-coloring.
Note that no edge of $K_5$ that connects two vertices of opposite colors is subdivided.

There are six bipartite graphs which are subdivisions of
$K_{3,3}$ to consider, denoted by $G_{(i,j)}$ and defined as follows.
Let $X$ and $Y$ be the two sides of $K_{3,3}$. Then
$G_{(i,j)}$ is the graph obtained from $K_{3,3}$ by first
(i) coloring red exactly $i$ vertices from $X$ and $j$ vertices from $Y$,
and coloring blue the other $6-i-j$ vertices;
then (ii) subdividing each monochromatic edge once,
and coloring the subdivision vertex red/blue as before,
so that a proper $2$-coloring is obtained. Note that as before, no edge of $K_{3,3}$
that connects two vertices of opposite colors is subdivided.
Up to symmetry, the six graphs we need to consider are
$G_{(3,3)}, G_{(3,2)}, G_{(3,1)}, G_{(3,0)}, G_{(2,2)}, G_{(2,1)}$.
(Observe that $G_{(3,0)}=K_{3,3}$, and so this case is trivial.)
\end{proof}

\begin{rem} It is worth noting that
%there is no bipartite analog of the
% Kuratowski theorem. Indeed,
the barycentric subdivision of $K_5$ (which is a bipartite graph)
does not contain subgraphs homeomorphic to $K_{3,3}$.
\end{rem}

A graph is \emph{outerplanar} if it can be embedded in the plane in
such a way that all of the vertices lie on the outer boundary.
Equivalently, a graph $G$ is outerplanar if adding a new vertex to $G$
and connecting it to all vertices of $G$ results in a planar graph;
we denote this graph by $\widehat{G}$.
%if it is a subgraph of a planar graph that admits a planar
%embedding where all vertices are on the boundary of a single face.
Outerplanar graphs are characterized by not having as a minor $K_4$
and $K_{2,3}$. Here is a bipartite analogue of this result for
bipartite minors; the proof is similar to the proof of
Theorem \ref{thm:bipWagner}.

\begin{theo}\label{thm:bipOuterplanar}
A bipartite graph $G$ is outerplanar if and only if $G$ does not
contain $K_{2,3}$ as a bipartite minor.
\end{theo}
\begin{proof} First assume that $G$ is outerplanar.
To verify that $G$ does not contain $K_{2,3}$ as a bipartite minor,
it suffices to show that deletions and admissible contractions preserve
outerplanarity. This is clear for deletions. To deal with admissible
contractions, consider the graph $\widehat{G}$
defined right before the statement of the theorem.
Then $\widehat{G}$ is planar, although not bipartite, and if
$G'$ is obtained from $G$ by an admissible contraction of $a'$ with $a$, then
$\widehat{G'}=(\widehat{G})'$. (Note that a peripheral cycle of $G$ is also
a peripheral cycle of $\widehat{G}$, so the same contraction is
admissible in $\widehat{G}$.) As $(\widehat{G})'$ is planar
(the same argument as in the proof of Theorem \ref{thm:bipWagner} applies),
we infer that $G'$ is outerplanar.

Next assume that $G$ is not outerplanar. By a result of Chartrand
and Harary \cite{Chartrand-Harary67:Outerplanar}, $G$ contains a
subdivision of either $K_4$ or $K_{2,3}$. Let $H$ be such a subgraph
of $G$. As in the proof of Theorem \ref{thm:bipWagner}, we may assume
that each edge of the original $K_4$ ($K_{2,3}$, respectively) is subdivided at most
once. Thus, it suffices to show that $K_{2,3}$ is a bipartite minor of
each of the following nine bipartite graphs.
Given the coloring below, we can find $K_{2,3}$ with three red
vertices and two blue ones as a bipartite minor.

For subdivisions of $K_4$, we need to consider $H=H_{(i)}$ for
$i=4,3,2$, where in $H_{(i)}$ exactly $i$ of the original vertices of
$K_4$ are red (the other $4-i$ vertices are blue). For subdivisions of
$K_{2,3}$, let $X$ and $Y$ be the two sides of $K_{2,3}$, with
$|X|=2$ and $|Y|=3$, and for the bipartite subdivisions
$H_{(i,j)}$ as above, with exactly $i$ red vertices from $X$ and
$j$ red vertices from $Y$, we need to consider $H$ being one of
$H_{(2,3)},H_{(1,3)},H_{(0,3)},H_{(2,2)},H_{(1,2)},H_{(2,1)}
$. We leave the verification of these nine cases to the readers.
\end{proof}

Similarly, the following holds; we omit an easy proof.

\begin{theo}\label{thm:bipForest}
A bipartite graph $G$ is a forest if and only if $G$ does not
contain $K_{2,2}$ as a bipartite minor.
\end{theo}

The following question arises naturally.

\begin{prob}
Are the bipartite linklessly embeddable graphs characterized by a
finite list of forbidden bipartite minors?
\end{prob}
%\eran{Is this question reasonable?}
%%%%%%%%%%%%%%%%%%%%%%%%%%%%%%%%%
%%%%%%%%%%%%%%%%%%%%%%%%%%%%%%%%%%%%
\section{$(2,2)$-Laman graphs}
We now turn our discussion to $(2,2)$-Laman graphs. Note that if
$G=(V=A\uplus B, E)$ is $(2,2)$-Laman, then every vertex of $G$ has degree
at least two: indeed, if $v$ were a vertex of degree one, then $G-v$
would have $2(|A|+|B|-1)-3$ edges instead of at most $2(|A|+|B|-1)-4$
edges allowed by the definition of $(2,2)$-Laman graphs. Moreover, if $G$ is
$(2,2)$-Laman and $v$ is a vertex of degree two, then either $G$ is $K_{2,2}$
or $G-v$ is also $(2,2)$-Laman. Finally,
since $G$ has fewer than $2|V|$ edges, there is a vertex of $G$
that has degree at most three. Hence, we can assume that $G$ is a graph with minimal
degree three. The following theorem can thus be considered as a recursive
characterization of $(2,2)$-Laman graphs.

\begin{theo}\label{prop:GraphTheory}
Let $G=(V,E)$ be a bipartite $(2,2)$-Laman graph with minimal degree three.
Then every vertex $v$ of degree three has two
neighbors $x,y$ with the property that there exists a vertex $p$ that is adjacent
to $y$ and not adjacent to $x$, and such that the graph
$G'=(G - v) \cup xp$ is $(2,2)$-Laman.
\end{theo}

\begin{proof}
If $X\subseteq V$, we write $E(X)=|E(G[X])|$ --- the cardinality of the
edge set of the subgraph of $G$ induced by $X$. A subset $X$ of $V$ is
\emph{critical} if $|X| \geq 3$ and $E(X)=2|X|-4$; equivalently, if
$|X|\geq 3$ and $G[X]$ is $(2,2)$-Laman or $K_{2,1}$.
%To save space, ``WMA'' stands for ``we may assume''.

Let $(A,B)$ be a bipartition of $V$.
In what follows vertices called $a_i$ belong to $A$, and vertices called
$b_i$ belong to $B$. Suppose that $a_0\in A$ is a vertex of degree $3$
and let $b_1,b_2,b_3$ be the neighbors of $a_0$.
We prove the theorem in several steps, which we number below by (i*).

\smallskip\noindent{\bf (0*)}
{\em Every subset $X\subseteq V$ with $|X|\geq 2$ such that $G[X]$
is not an edge satisfies $E(X)\leq 2|X|-4$.}

\begin{proof}
This is immediate from the definition of $(2,2)$-Laman graphs.
\end{proof}

\smallskip\noindent{\bf (1*)}
{\em At least two neighbors of $a_0$ have non-neighbors in $A$.}

\begin{proof}
The subgraph induced on $A\cup \{b_1,b_2,b_3\}$ has $|A|+3$ vertices and
hence at most $2|A|+2$ edges. Thus if $b_1,b_2$ are adjacent to all of $A$
then $b_3$ has degree at most two, a contradiction.
\end{proof}

Let $b_1,b_2$ be as guaranteed in (1*).
Let $Z\subseteq V$ be maximal with $2|Z|-4$ edges, containing $b_1,b_2$ and
not containing $a_0$ (possibly $Z=\{b_1,b_2\}$).

\smallskip\noindent{\bf (2*)}
{\em The element $b_3$ is not in $Z$ and has at most one neighbor in $Z$.
In particular, $b_3$ has a neighbor in $A\setminus (Z\cup\{a_0\})$, and so the
latter set is nonempty.}

\begin{proof}
If $b_3\in Z$ then $Z\cup\{a_0\}$ violates Laman condition (ii), a
contradiction, and so $b_3\notin Z$. Now $G[Z\cup\{a_0\}]$ is $(2,2)$-Laman,
hence the Laman condition (ii) for $Z\cup \{a_0,b_3\}$ shows that $b_3$ has
at most one neighbor in $Z$. As $\deg(b_3)\geq 3$ the rest of (2*) follows.
\end{proof}

Denote by $M(b_i)$ the set of non-neighbors of $b_i$ in
$A\setminus (Z\cup \{a_0\})$, and w.l.o.g.~assume $|M(b_1)|\geq |M(b_2)|$.
Note that the maximality of $Z$ implies that

\smallskip\noindent{\bf (3*)} {\em
Every vertex in $A\setminus (Z\cup \{a_0\})$ has at most one neighbor in $Z$.}
%$\square$

In particular, we will use the following:

\smallskip\noindent{\bf (4*)}
{\em No element in $A\setminus (Z\cup \{a_0\})$ is a neighbor of both
$b_1$ and $b_2$.}
%$\square$

%\smallskip\noindent{\bf (3**)} {\em $M(b_1)\neq \emptyset$.}

%\begin{proof}
%Assume the contrary. Then all elements of $A\setminus (Z\cup \{a_0\})$ are
%neighbors of both $b_1$ and $b_2$. Thus

%$Z\cup (A\setminus (Z\cup \{a_0\}))$
%has $2|Z\cup (A\setminus (Z\cup \{a_0\}))|-4$ edges. By the maximality of
%$Z$, $A\setminus (Z\cup \{a_0\}) \subseteq Z$, and so

%$A\setminus (Z\cup
%\{a_0\})=\emptyset$, contradicting (2*).
%\end{proof}

%The argument above on maximality of $Z$ also implies:

\smallskip\noindent{\bf (5*)}
{\em The set $A_1$ of neighbors of either $b_2$ or $b_3$ in $M(b_1)$ is nonempty.}
\begin{proof}
Either there exists $p \in M(b_1)\setminus M(b_2)$ (that is, $p$ is a
neighbor of $b_2$ but not of $b_1$), in which case we are done, or
$M(b_1)=M(b_2)$. In the latter case
$M(b_1)=A\setminus (Z\cup \{a_0\})$ by (4*).
Hence by (2*), there is a vertex $p\in M(b_1)$ that is a neighbor of
$b_3$. The statement follows.
\end{proof}
%Thus,
%\begin{eqnarray}\label{eqn:A1}
%\text{ The set $A_1$ of neighbors of either $b_2$ or $b_3$ in $M(b_1)$ is nonempty.}
%\end{eqnarray}

We need to prove that there is $p\in A_1$ such that
no critical set contains $p$ and $b_1$ and not $a_0$.
Assume the contrary.
Then for every $p\in A_1$ there is some critical set $X_p$
containing $p$ and $b_1$, and not containing $a_0$. We will reach a
contradiction to (5*). We may assume that the sets $X_p$
are maximal with these properties.
%
%and number them $X_1,\cdots,X_s$ (not repeating the same set twice).

\smallskip\noindent{\bf (6*)} {\em If $p\in A_1$ then $b_2\notin X_p$.}

\begin{proof}
Assume by contradiction $b_2\in X_p$. By maximality of $Z$ and as
$p\notin Z$, it is enough to show that $X_p\cup Z$ is critical. Indeed,
\begin{eqnarray*}
\begin{array}{c}
E(Z\cup X_p) \geq E(Z)+E(X_p)-E(Z\cap X_p) = \\
2(|Z|+|X_p|)-8-E(Z\cap X_p) \geq 2(|Z|+|X_p|)-8-(2|Z\cap X_p|-4) \geq \\
2|Z\cup X_p|-4,
\end{array}
\end{eqnarray*}
where the middle inequality is by (0*), as $\{b_1,b_2\}\subseteq Z\cap X_p$
are two vertices on the same side.
\end{proof}

\smallskip\noindent{\bf (7*)} {\em For every $p\in A_1$,
$b_1$ has at least two neighbors in $X_p$.}

\begin{proof}
Since $X_p$ is critical containing two non-adjacent vertices from
opposite sides (namely $p$ and $b_1$), we deduce that $|X_p| \geq 4$. Then
$E(X_p \setminus b_1) \leq 2(|X_p|-1)-4$ by Laman condition (ii), and so
$b_1$ has at least two neighbors in $X_p$.
\end{proof}

Let $X_0\subseteq V$ (in our notations $0\notin V$) consist of
$b_1$ and its neighbors different from $a_0$.
For notational convenience, identify $A_1$ with $\{1,2,\ldots,s\}$, and
let $Y_s=X_0\cup X_1\cup\cdots\cup X_s$. Let $k=|X_0|-1$, namely the
degree of $b_1$ in $G[V\setminus a_0]$ (hence $k\geq 2$).

\smallskip\noindent{\bf (8*)} $E(Y_s)\geq 2|Y_s|-4-(k-2)$.

\begin{proof}
 Let $Y_i=\cup_{0\leq j\leq i}X_j$. We show by induction on $i$ that
$E(Y_i)\geq 2|Y_i|-4-(k-2)$. The case $i=0$ trivially holds with equality.
For $i\geq 1$, $|X_0\cap X_i|\geq 3$ by (7*), so applying Laman condition
(ii) to $Y_{i-1}\cap X_i$ we see that
\begin{eqnarray*}
\begin{array}{c}
E(Y_i) \geq E(Y_{i-1})+E(X_i)-E(Y_{i-1}\cap X_i) \geq \\
(2|Y_{i-1}|-4-(k-2)) + (2|X_i|-4) - (2|Y_{i-1}\cap X_i|-4) = \\
2|Y_{i-1}\cup X_i|-4-(k-2),
\end{array}
\end{eqnarray*}
as desired.
\end{proof}

Next we will show that $Z\cup Y_s$ is critical.

\smallskip\noindent{\bf (9*)}
{\em $Z\cap Y_s \subseteq X_0$, and thus $Z\cap Y_s = Z\cap X_0$.}

\begin{proof}
 Assume the contrary. Then there is $i$ such that $Z\cap X_i$ has a
non-neighbor of $b_1$. Hence by (0*), $E(Z\cap X_i)\leq 2|Z\cap X_i|-4$.
Thus, our usual yoga shows $E(Z\cup X_i)\geq 2|Z\cup X_i|-4$. By maximality
of $X_i$, we conclude that $Z\subseteq X_i$, and so $b_2\in X_i$.
This contradicts (6*).
\end{proof}

\smallskip\noindent{\bf (10*)} {\em $Z\cup Y_s$ is critical.}

\begin{proof}
Denote by $d_i$ the number of neighbors of $b_i$ in $A\setminus
(Z\cup\{a_0\})$.
 By (9*), $E(Z\cap Y_s)=k-d_1$ and $|Z\cap Y_s|=k-d_1+1$.
 There are $d_2$ edges from $b_2\in Z$ into $A_1\subseteq Y_s\setminus Z$ by (4*).
Therefore,
\begin{eqnarray*}
\begin{array}{c}
E(Z\cup Y_s)  \geq E(Z)+E(Y_s)-E(Z\cap Y_s)+d_2  \geq \\
(2|Z|-4)+(2|Y_s|-4-(k-2))-(k-d_1)+d_2  = \\
2|Z\cup Y_s|-4+(d_2-d_1),
\end{array}
\end{eqnarray*}
and by $|M(b_1)|\geq |M(b_2)|$  we have $d_2-d_1\geq 0$. Thus $Z\cup Y_s$
is critical. \end{proof}

By maximality of $Z$, (10*) implies $Y_s\subseteq Z$. Thus $A_1\subseteq Z$.
Hence by the definition of $A_1$, $A_1$ must be empty. This contradicts (5*) and
completes the proof of the theorem.
\end{proof}

%%%%%%%%%%%%%%%%%%%%%%%%%%
\bibliography{gbiblio}
\bibliographystyle{plain}

\section{Appendix}
For each of the nine graphs described in the proof of Theorem \ref{thm:bipWagner},
we describe the sequences of deletions and admissible contractions yielding
$K_{3,3}$ as a bipartite minor. For each contraction we indicate a
peripheral cycle showing that the contraction is admissible,
called a \emph{witness cycle}.

We start with the subdivisions of $K_5$.
Denote the vertices of $K_5$ by $v_i$ where $i\in[5]$
(here $[n]=[1,n]=\{1,2,\cdots,n\}$), and the subdivision vertex
of the edge $v_iv_j$ in $G_{(i)}$ by $v_{ij}=v_{ji}$. When referring
to vertices after performing contractions on $G_{(i)}$, we use
\emph{any} representative from the vertices of $G_{(i)}$;
this should cause no confusion.

\smallskip\noindent\textbf{Case 1: $G_{(5)}$.}
Here all $v_i$ where $i\in [5]$ are red.
\begin{enumerate}
\item{}Contract $v_{15}$ with $v_{13}$.
Witness cycle: $(v_{15} v_1 v_{13}, v_3 v_{35} v_5)$.
\item{}Contract $v_{25}$ with $v_{23}$.
Witness cycle: $(v_{25} v_2 v_{23}, v_3 v_{35} v_5)$.
\item{}Contract $v_{45}$ with $v_{14}$.
Witness cycle: $(v_{45} v_4 v_{14}, v_1 v_{15} v_5)$.
\item{}Contract $v_{1}$ with $v_{2}$.
Witness cycle: $(v_{1} v_{12} v_{2}, v_{25} v_{5} v_{15})$.
\item{}Contract $v_{3}$ with $v_{4}$.
Witness cycle: $(v_{4} v_{34} v_{3}, v_{35} v_{5} v_{54})$.
\end{enumerate}
Call the resulting graph $H$.
The induced subgraph of $H$ on the $3$ red vertices $v_1,v_3,v_5$
and the $3$ blue vertices $v_{15},v_{25},v_{45}$ is $K_{3,3}$.

\smallskip\noindent\textbf{Case 2: $G_{(4)}$.}
Here all $v_i$ where $i\in [4]$ are red, $v_5$ is blue.
\begin{enumerate}
\item{}Contract $v_{34}$ with $v_{23}$.
Witness cycle: $(v_{34} v_3 v_{23}, v_2 v_{24} v_4)$.
\item{}Contract $v_{12}$ with $v_{14}$.
Witness cycle: $(v_{12} v_1 v_{14}, v_4 v_{24} v_2)$.
\item{}Contract $v_{12}$ with $v_{13}$.
Witness cycle: $(v_{12} v_1 v_{13}, v_3 v_{32} v_2)$.
\end{enumerate}
Call the resulting graph $H$.
The induced subgraph of $H$ on the $3$ red vertices $v_2,v_3,v_4$
and the $3$ blue vertices $v_{5},v_{34},v_{12}$ is $K_{3,3}$.

\smallskip\noindent\textbf{Case 3: $G_{(2)}$.}
Here all $v_i$ where $i\in [2]$ are red,
the other $3$ vertices of $G_{(2)}$ are blue.
\begin{enumerate}
\item{}Contract $v_{34}$ with $v_{35}$.
Witness cycle: $(v_{34} v_3 v_{35}, v_5 v_{45} v_4)$.
\end{enumerate}
Call the resulting graph $H$.
The induced subgraph of $H$ on the $3$ red vertices $v_1,v_2,v_{34}$
and the $3$ blue vertices $v_{3},v_{4},v_{5}$ is $K_{3,3}$.

\bigskip

We now turn to the six subdivisions of $K_{3,3}$.
Let the sides of $K_{3,3}$ be $X=\{v_1,v_2,v_3\}$ and
$Y=\{v_4,v_5,v_6\}$. We use the same notation $v_{ij}$ as in the case
of subdivisions of $K_5$.
%%%%%%%%%%%%%%%%%%%%%%%%%%%%%%%

\smallskip\noindent\textbf{Case 1: $G_{(3,3)}$.}
Here all $v_i$ where $i\in [6]$ are red.
\begin{enumerate}
\item{}Contract $v_{15}$ with $v_{35}$.
Witness cycle: $(v_{15} v_5 v_{35}, v_3 v_{34} v_4 v_{14} v_1)$.
\item{}Contract $v_{14}$ with $v_{24}$.
Witness cycle: $(v_{14} v_4 v_{24}, v_2 v_{26} v_6 v_{16} v_1)$.
\item{}Contract $v_{26}$ with $v_{36}$.
Witness cycle: $(v_{26} v_6 v_{36}, v_3 v_{34} v_4 v_{24} v_2)$.
\item{}Contract $v_{1}$ with $v_{6}$.
Witness cycle: $(v_{1} v_{16} v_{6}, v_{36} v_{2} v_{24})$.
\item{}Contract $v_{2}$ with $v_{5}$.
Witness cycle: $(v_{2} v_{25} v_{5}, v_{35} v_{3} v_{34} v_4 v_{24})$.
\item{}Delete $v_{16}$, then delete $v_{25}$.
\item{}Contract $v_3$ with $v_4$.
Witness cycle:  $(v_{3} v_{34} v_{4}, v_{24} v_{1} v_{36})$.
\end{enumerate}
Call the resulting graph $H$.
The induced subgraph of $H$ on the $3$ red vertices $v_1,v_2,v_3$
and the $3$ blue vertices $v_{15},v_{14},v_{26}$ is $K_{3,3}$.

\smallskip\noindent\textbf{Case 2: $G_{(3,2)}$.}
Here all $v_i$ where $i\in [5]$ are red and $v_6$ is blue.
\begin{enumerate}
\item{}Contract $v_{15}$ with $v_{35}$.
Witness cycle: $(v_{15} v_5 v_{35}, v_3 v_{34} v_4 v_{14} v_1)$.
\item{}Contract $v_{14}$ with $v_{24}$.
Witness cycle: $(v_{14} v_4 v_{24}, v_2 v_{25} v_5 v_{15} v_1)$.
\item{}Contract $v_{3}$ with $v_{4}$.
Witness cycle: $(v_{3} v_{34} v_{4}, v_{24} v_{2} v_{25} v_5 v_{15})$.
\item{}Contract $v_{2}$ with $v_{5}$.
Witness cycle: $(v_{2} v_{25} v_{5}, v_{15} v_{1} v_6)$.
\end{enumerate}
Call the resulting graph $H$.
The induced subgraph of $H$ on the $3$ red vertices $v_1,v_2,v_3$
and the $3$ blue vertices $v_{15},v_{14},v_{6}$ is $K_{3,3}$.

\smallskip\noindent\textbf{Case 3: $G_{(3,1)}$.}
Here all $v_i$ where $i\in [4]$ are red and $v_5, v_6$ are blue.
\begin{enumerate}
\item{}Contract $v_{24}$ with $v_{34}$.
Witness cycle: $(v_{24} v_4 v_{34}, v_3 v_{5} v_2)$.
\item{}Contract $v_{1}$ with $v_{4}$.
Witness cycle: $(v_{1} v_{14} v_{4}, v_{34} v_3 v_{5})$.
\end{enumerate}
Call the resulting graph $H$.
The induced subgraph of $H$ on the $3$ red vertices $v_1,v_2,v_3$
and the $3$ blue vertices $v_{24},v_{5},v_{6}$ is $K_{3,3}$.

\smallskip\noindent\textbf{Case 4: $G_{(2,2)}$.}
Here all $v_i$ where $i\in [2,5]$ are red and $v_1, v_6$ are blue.
\begin{enumerate}
\item{}Contract $v_{24}$ with $v_{34}$.
Witness cycle: $(v_{34} v_4 v_{24}, v_2 v_{25} v_5 v_{53} v_3)$.
\item{}Contract $v_{25}$ with $v_{35}$.
Witness cycle: $(v_{25} v_5 v_{35}, v_3 v_{34} v_2)$.
\item{}Contract $v_{4}$ with $v_{5}$.
Witness cycle: $(v_{4} v_{1} v_{5}, v_{25} v_3 v_{34})$.
\item{}Contract $v_{1}$ with $v_{6}$.
Witness cycle: $(v_{1} v_{16} v_{6}, v_{2} v_{25} v_{4})$.
\end{enumerate}
Call the resulting graph $H$.
The induced subgraph of $H$ on the $3$ red vertices $v_2,v_3,v_4$
and the $3$ blue vertices $v_{24},v_{25},v_{6}$ is $K_{3,3}$.

\smallskip\noindent\textbf{Case 5: $G_{(2,1)}$.}
Here all $v_i$ where $i\in [2,4]$ are red and $v_1, v_5, v_6$ are blue.
\begin{enumerate}
\item{}Contract $v_{16}$ with $v_{15}$.
Witness cycle: $(v_{16} v_1 v_{15}, v_5 v_{3} v_6)$.
\item{}Contract $v_{24}$ with $v_{34}$.
Witness cycle: $(v_{24} v_4 v_{34}, v_3 v_{6} v_2)$.
\item{}Contract $v_{1}$ with $v_{34}$.
Witness cycle: $(v_{1} v_{4} v_{34}, v_{2} v_5 v_{15})$.
\end{enumerate}
Call the resulting graph $H$.
The induced subgraph of $H$ on the $3$ red vertices $v_{16},v_2,v_3$
and the $3$ blue vertices $v_{1},v_{5},v_{6}$ is $K_{3,3}$.

\smallskip\noindent\textbf{Case 6: $G_{(3,0)}$.}
In this case $G_{(3,0)}=H=K_{3,3}$.
This completes the proof of Theorem \ref{thm:bipWagner}. $\square$

\end{document}